\documentclass[11pt]{amsart}
\usepackage{amsfonts}
\usepackage{amsmath}
\usepackage{amssymb}
\usepackage[T1]{fontenc}
\usepackage{textcomp}

\newtheorem{theorem}{Theorem}[section]

\newtheorem{corollary}[theorem]{Corollary}

\newtheorem{example}[theorem]{Example}

\newtheorem{lemma}[theorem]{Lemma}

\newtheorem{proposition}[theorem]{Proposition}
\newtheorem{remark}[theorem]{Remark}

\numberwithin{equation}{section}

\begin{document}
\title{Liouville-type theorems on the complete gradient shrinking Ricci solitons}
\author{Huabin Ge}
\address{Huabin Ge, Department of Mathematics,
Beijing Jiaotong University,
Beijing, 100044,
P. R. China, hbge@bjtu.edu.cn}

\author{Shijin Zhang}
\address{Shijin Zhang, School of Mathematics and systems science, Beihang University, Beijing,  100871, P.R.China, shijinzhang@buaa.edu.cn}


\maketitle

\begin{abstract}
We prove that there does not exist non-constant positive $f$-harmonic function on the complete gradient shrinking Ricci solitons. We also prove the $L^{p}(p\geq 1 \  {\rm or} \ 0<p\leq 1)$ Liouville theorems on the complete gradient shrinking Ricci solitons.
\end{abstract}

\vskip6mm
\section*{Introduction}
Gradient Ricci solitons play an important role in Hamilton's Ricci flow as they correspond to self-similar solutions, and often arise as singularity models.  The study of solitons has also become increasingly important in the study in metric measure theory.

In this paper we study the properties of complete gradient shrinking Ricci solitons and positive $f$-harmonic (or $f$-subharmonic, $f$-superharmonic) functions on complete gradient shrinking Ricci solitons. A complete Riemannian manifold $(M,g)$ is called a gradient Ricci soliton if there exists a smooth function $f$ on $M$ such that
\begin{equation*}
{\rm Ric}+\nabla\nabla f=\lambda g
\end{equation*}
for some constant $\lambda$. For $\lambda<0$ the Ricci soliton is expanding, for $\lambda=0$ it is steady and for $\lambda >0$ is shrinking. The function $f$ is called a potential function of the gradient Ricci soliton. After rescaling the metric $g$ we may assume that $\lambda\in\{-\frac{1}{2},0,\frac{1}{2}\}$.

Bakry-\'Emery Ricci tensor on $(M,g)$ is defined by
\begin{equation*}
{\rm Ric}_{f}={\rm Ric}+\nabla\nabla f
\end{equation*}
for some smooth function $f$ on $M$. It is an extension of Ricci tensor. $(M,g)$ is a gradient Ricci soliton if
$${\rm Ric}_{f}=\lambda g$$
for some constant $\lambda$ and some smooth function $f$ on $M$. The $f$-Laplacian operator defined for a function $u$ by
\begin{equation*}
\Delta_{f}u=\Delta u-\langle\nabla f,\nabla u\rangle.
\end{equation*}
If $f$ is constant, the $f$-Laplacian operator is the Laplace-Beltrami operator. A smooth function $u$ on $M$ is called a $f$-harmonic function if $\Delta_{f} u=0$; a  smooth function $u$ on $M$ is called a $f$-subharmonic function if $\Delta_{f} u\geq 0$; a smooth function $u$ on $M$ is called a $f$-superharmonic function if $\Delta_{f} u\leq 0$. If $f$ is constant, then the $f$-harmonic function, $f$-subharmonic function, $f$-superharmonic function is the harmonic function, subharmonic function, superharmonic function, respectively. Hence a $f$-harmonic (or $f$-subharmonic, $f$-superharmonic) function can look as the generalization of harmonic (or subharmonic, superharmonic) function. For a positive harmonic function $u$ on a Riemannian manifold $(M,g)$ with Ricci curvature bounded from below, i.e., ${\rm Ric}\geq -(n-1)K$ for some constant $K\geq 0$, Yau (see Theorem $3''$ in \cite{Yau}) has obtain a global gradient estimate
\begin{equation*}
|\nabla u|\leq (n-1)\sqrt{K}u.
\end{equation*}
Then Yau obtained the Liouville theorem, there does not exist nonconstant positive harmonic function on complete Riemannian manifold with nonnegative Ricci curvature. Later, Cheng and Yau (\cite{Cheng-Yau}, Theorem 6), obtained a local gradient estimate for positive harmonic function on complete Riemannian manifold with Ricci curvature bounded from below. It stated that for $u: B_{p}(2R)\rightarrow \mathbb{R}$ harmonic and positive, if the Ricci curvature on $B_{p}(2R)$ has a lower bound ${\rm Ric}\geq -(n-1)K$ for some $K\geq 0$, then
\begin{equation}
\sup_{B_{p}(R)}|\nabla \log u|\leq (n-1)\sqrt{K}+\frac{C}{R}.
\end{equation}
Then if we take $R\rightarrow+\infty$, it is Yau's global gradient estimate. For the case of $K>0$, recently, Munteanu (\cite{Munteanu}) has improved the above gradient estimate as following
\begin{equation*}
\sup_{B_{p}(R)}|\nabla \log u|\leq (n-1)\sqrt{K}+\frac{C}{R}e^{-CR}
\end{equation*}
for some constant $C>0$.

For the complete smooth measure space $(M,g, e^{-f}dv)$ with ${\rm Ric}_{f}\geq -(n-1)K$. If $u$ is a $f$-harmonic function on $B_{p}(2R)$ with $R\geq 1$, Brighton (\cite{B}) proved the following gradient estimate
\begin{equation*}
\sup_{B_{p}(R)}|\nabla u|\leq \sqrt{\frac{c_{1}(n)}{R}+c_{2}(n)K}\sup_{B_{p}(2R)}u
\end{equation*}
for some constants $c_{1}(n)$ and $c_{2}(n)$. Using this gradient estimate, Brighton (\cite{B}) proved that a positive bounded $f$-harmonic function on a complete smooth measure space $(M,g, e^{-f}dv)$ with nonnegative Bakry-\'Emery Ricci curvature is constant.

For a local Cheng-Yau's gradient estimate for Bakry-\'Emery Ricci tensor bounded from below ${\rm Ric}_{f}\geq -(n-1)K$ and $|\nabla f|\leq \theta$, Wu (\cite{Wu2}) proved that
a positive $f$-harmonic function on $B_{p}(2R)$ has the following local gradient estimate
\begin{equation*}
\sup_{B_{p}(R)}\frac{|\nabla u|}{u}\leq \sqrt{c_{1}(n)K+c_{2}(n)\theta^{2}+\frac{c_{3}(n)}{R^{2}}}
\end{equation*}
for some positive constants $c_{1}(n), c_{2}(n), c_{3}(n)$. Chen and Chen (\cite{Chen-Chen}) also proved the same gradient estimate for positive $f$-harmonic function with another condition ${\rm Ric}\geq -(n-1)H$.  Munteanu and Wang \cite{Munteanu-WangJ} also proved a global gradient estimate for positive $f$-harmonic function, they proved that a positive $f$-harmonic function with sublinear growth of $f$ on complete Riemannian manifold with ${\rm Ric}_{f}\geq 0$ is constant.

If we consider $\theta$ in the condition of Wu's local gradient estimate as a function of radius, we also prove a gradient estimate for the smooth metric measure space.
\begin{theorem}\label{LGEMS}
Let $(M,g, e^{-f}dv)$ be a smooth metric measure space with $Ric_{f}\geq-(n-1)K$ for $K\geq 0$. Fixed point $p\in M$. Assume $u$ is a positive $f$-harmonic function on $B_{p}(2R)$. Denote
\begin{equation*}
\theta(R)\doteqdot \sup_{B_{p}(2R)}|\nabla f|.
\end{equation*}
Then for any $0<\epsilon<1$, we have
\begin{eqnarray*}
\begin{aligned}
\sup_{B_{p}(R)}|\nabla \log u|&\leq \frac{c(n)}{1-\epsilon}(\frac{1}{R}+\frac{K^{1/4}+\theta(R)^{1/2}}{R^{1/2}})+\frac{1}{\sqrt{\epsilon(1-\epsilon)}}\theta(R)\\
&\quad+\frac{(n-1)\sqrt{K}}{\sqrt{1-\epsilon}}.
\end{aligned}
\end{eqnarray*}
\end{theorem}

Since the potential functional on the complete shrinking Ricci soliton is equivalent to the distance function, hence we have
\begin{theorem}\label{LocalGradientEstimate}
Let $(M,g,f)$ be a complete gradient shrinking Ricci soliton, then there exists a constant $C>0$, if $u$ is a positive $f$-harmonic function on $B_{p}(3R)$ for $R>C$, we have
\begin{equation}
\sup_{x\in B_{p}(R)}|\nabla \log u|(x)\leq \frac{20n}{R}+2R+3C.
\end{equation}
\end{theorem}
As a consequence, we obtain the Harnack inequality.
\begin{corollary}\label{HarnackInequality}
Let $(M,g,f)$ be a complete shrinking gradient Ricci soliton, fix a point $p \in M$. Assume $u$ is a positive $f$-harmonic function on $M$, then we have
\begin{equation}
u(x)\leq u(p)e^{3d^{2}(x)}
\end{equation}
for $d(x)\geq \max\{2C, 10n\}$. Here $d(x)$ denotes the distance function from $x$ to $p$, $C$ is the constant in Theorem \ref{LocalGradientEstimate}.
\end{corollary}

 We also obtain some global estimates on the complete gradient shrinking Ricci solitons.  For the complete gradient shrinking Ricci solitons, we consider the potential function instead of distance function in the cutoff function, then we show that the nonnegative $L^{1}(e^{-f}dV_{g})$ $f$-superharmonic or $f$-subharmonic function   must be a $f$-harmonic function.
\begin{theorem}\label{DivergenceThm}
Let $(M,g,f)$ be a complete gradient shrinking Ricci soliton, $u$ is a nonnegative smooth functions on $M$. If
\begin{equation}
\Delta_{f}u\geq 0
\end{equation}
and
\begin{equation}
\int_{M}u e^{-f}dv <+\infty,
\end{equation}
then
\begin{equation}
\Delta_{f}u=0.
\end{equation}
\end{theorem}

Then using the Corollary \ref{HarnackInequality} and Theorem \ref{DivergenceThm}, we can prove the following Liouville-Type theorem on the complete shrinking Ricci solitons, without the bounded condition of the $f$-harmonic function.
\begin{theorem}\label{MainThm}
Let $(M,g,f)$ be a complete gradient shrinking Ricci soliton , i.e.,
\begin{equation*}
{\rm Ric}+\nabla\nabla f=\frac{1}{2}g,
\end{equation*}
then any positive $f$-harmonic function on $M$ is constant.
\end{theorem}

\begin{remark}
For the metric measure space $(M,g, e^{-f}dv)$, Wei and Wylie \cite{Wei-Wylie} proved that if ${\rm Ric}_{f}\geq \lambda>0$, $u>0$, $\Delta_{f}u\geq 0$ and there is a constant $\alpha<\lambda$ such that $u(x)\leq e^{\alpha d^{2}(x)}$, then $u$ is constant.
\end{remark}

\begin{remark}\label{Jiayong}
The assumption of $f$-harmonic function is positive (or is bounded from below) is necessary. If we remove the condition, the result is false. One easy example is provided by Wu-Wu (see Example 1.8 in \cite{Wu-Wu}), as following
\begin{example}
Consider the $1$-dimensional Gaussian soliton $(\mathbb{R}, g, f)$, here $g$ is the Euclidean metric and $f(x)=\frac{x^2}{4}, u(x)=\int_{0}^{x}e^{t^2/4}dt$, then $u(x)$ is a $f$-harmonic function, and $u(x)\rightarrow -\infty$ as $x\rightarrow -\infty$.
\end{example}
\end{remark}

Using Theorem \ref{DivergenceThm} and Theorem \ref{MainThm}, we can obtain $L^{p}$ Liouville theorems on the complete shrinking gradient Ricci soliton.

\begin{corollary}\label{LpLiouvilleThm}
Let $(M,g,f)$ be a complete gradient shrinking Ricci soliton  and $p\geq 1$. Assume $u$ is a nonnegative $f$-subharmonic function, i.e., $\Delta_{f}u\geq 0$, and satisfies $u\in L^{p}(e^{-f}dv)$, i.e.,
$$\int_{M}u^{p}e^{-f}dv<+\infty, $$
then $u$ is constant.
\end{corollary}

\begin{remark}
For any complete smooth metric measure space $(M,g, e^{-f}dv)$, Pigola, Rimoldi and Setti proved the $L^{p}$ Liouville theorem for $p>1$, see \cite{Pigola-Rimoldi-Setti1}. For the case of $p=1$, Wu \cite{Wu2} proved the $L^{1}$ Liouville theorem under the assumptions $f$ is bounded and ${\rm Ric}_{f}\geq -C(1+d^{2}(x))$; without any assumption, is proved by Wu-Wu (see Corollary 1.6 in \cite{Wu-Wu}).
\end{remark}

\begin{corollary}\label{LpLiouvilleThm1}
Let $(M,g,f)$ be a complete gradient shrinking Ricci soliton  and $0<p\leq 1$. Assume $u$ is a positive $f$-superharmonic function, i.e., $\Delta_{f}u\leq 0$, and satisfies $u\in L^{p}(e^{-f}dv)$, i.e.,
$$\int_{M}u^{p}e^{-f}dv<+\infty, $$
then $u$ is constant.
\end{corollary}
The Corollary \ref{LpLiouvilleThm1} was proved by Pigola-Rimoldi-Setti, see Theorem 25 in \cite{Pigola-Rimoldi-Setti2}.

In Section 1 we recall some properties of the complete gradient shrinking Ricci solitons and prove Theorem \ref{DivergenceThm}; In Section 2, we obtain the gradient estimate on the complete gradient Ricci solitons, i.e., Theorem \ref{LGEMS}, then using the property of the complete gradient shrinking Ricci solitons to prove Theorem \ref{LocalGradientEstimate} and Corollary \ref{HarnackInequality}; In section 3, we prove Theorem \ref{MainThm}, Corollary \ref{LpLiouvilleThm} and Corollary \ref{LpLiouvilleThm1}.

 After we completed the first draft version of this paper, we learned that Ma \cite{Ma} also obtained the local gradient estimate on the complete gradient shrinking Ricci solitons, the Theorem \ref{MainThm} using the standard cutoff function and the Liouville Theorem on the complete gradient Ricci solitons with finite weighted energy.

\section{Proof of Theorem \ref{DivergenceThm}}
In this section, we will prove Themrem \ref{DivergenceThm}. First, we recall some properties of the gradient shrinking Ricci solitons.
\begin{lemma}
\quad
\begin{enumerate}
\item $S+\Delta f=\frac{n}{2}$, here $S$ means the scalar curvature;
\item $S+|\nabla f|^{2}=f$, after normalizing the function $f$;
\item For any fixed point $p\in M^{n}$, there exist two positive constants $c_{1}$ and $c_{2}$ so that, for any $x\in M^{n}$, we have
\begin{equation}
\frac{1}{4}(d(x)-c_{1})^{2}\leq f(x)\leq \frac{1}{4}(d(x)+c_{2})^{2},
\end{equation}
where $d(x)$ is the distance function from $x$ to $p$.
\end{enumerate}
\end{lemma}
(3) in the above lemma is proved by Cao and Zhou \cite{Cao-Zhou}(see also Fang, Man and Zhang \cite{Fang-Man-Zhang} and for an improvement, Haslhofer and M\"{u}ller \cite{Haslhofer-Muller}). It is well known that a complete gradient shrinking Ricci solitons has nonnegative scalar curvature (See Chen \cite{Chen}) and either $S>0$ or the metric $g$ is flat (see Pigola, Rimoldi and Setti \cite{Pigola-Rimoldi-Setti2} or the author \cite{Zhang}). Recently, Chow, Lu and Yang \cite{Chow-Lu-Yang} proved that the scalar curvature of a complete noncompact nonflat shrinker has a lower bound by $Cd(x)^{-2}$ for some positive constant $C$.

From the above lemma, we have $|\nabla f|^{2}\leq f$, $-f\leq \Delta f \leq \frac{n}{2}$. Then $2\sqrt{f}$ looks like as a the distance function, hence we use $2\sqrt{f}$ as the cut-off function on the complete gradient shrinking Ricci solitons.  We denote $\rho(x)=2\sqrt{f(x)}$, then we have
\begin{equation}
\nabla\rho=\frac{\nabla f}{\sqrt{f}}, \quad \Delta\rho=\frac{\Delta f}{\sqrt{f}}-\frac{|\nabla f|^{2}}{2f^{3/2}}, \quad \Delta_{f}\rho=\frac{\frac{n}{2}-f}{\sqrt{f}}-\frac{|\nabla f|^{2}}{2f^{3/2}}.
\end{equation}
Then we have
\begin{equation}\label{EstimateRho}
|\nabla \rho|\leq 1, \quad -\sqrt{f}\leq \Delta_{f}\rho \leq \frac{n}{2\sqrt{f}}.
\end{equation}
Now we prove the divergence theorem on the complete shrinking gradient Ricci solitons.
\begin{proof}[Proof of Theorem \ref{DivergenceThm}]
We choose a cutoff function $\eta:[0,\infty)\rightarrow [0,1]$ such that $\eta(t)=1$ for $0\leq t\leq 1$, $\eta(t)=0$ for $t\geq 2$ and $0\leq-\eta'(t)\leq 4, |\eta''(t)|\leq 4$. Let
\begin{equation}
\phi(x)=\eta(\frac{\rho(x)}{r}).
\end{equation}
Then it is easy to compute
$$\nabla \phi(x)=\frac{\eta'}{r}\nabla\rho(x),$$
$$\Delta_{f} \phi(x)=\frac{\eta'}{r}\Delta_{f}(\rho)+\frac{\eta''}{r^{2}}|\nabla\rho|^{2}.$$
Since $\nabla \phi $ and $\Delta_{f} \phi$ may not vanish only if $r\leq \rho(x)\leq 2r$. If $r\geq 2n$, using (\ref{EstimateRho}), we have
$|\nabla \phi(x)|\leq \frac{4}{r}$, $|\Delta_{f} \phi(x)|\leq 4+\frac{4}{r^{2}}$. Then
\begin{equation*}
\int_{M}(\Delta_{f}u)\phi e^{-f}dv=\int_{M}u\Delta_{f}\phi e^{-f}dv
\end{equation*}
Since for large $r$, $\{x: r\leq \rho(x)\leq 2r\}\subseteqq \{r-c_{2}\leq d(x)\leq 2r+c_{1}\}$, we have
\begin{equation*}
|\int_{M}(\Delta_{f}u)\phi e^{-f}dv|\leq \int_{r-c_{2}\leq d(x)\leq 2r+c_{1}}u(x)(4+\frac{4}{r^{2}})e^{-f}dv.
\end{equation*}
Since $\Delta_{f}u\leq 0$ (or $\Delta_{f}u\geq 0$) and $\int_{M}ue^{-f}dv<+\infty$, let $r\rightarrow \infty$, then the right hand side of the above inequality is $0$. Hence we obtain $\int_{M}\Delta_{f}(u) e^{-f}dv=0$. So $\Delta_{f}u=0$.
\end{proof}

\section{Gradient estimate on complete shrinkers}
In this section, we prove a local version of gradient estimate for positive f-harmonic function on $M$.

First, we show the gradient estimate for the smooth metric measure space.
\begin{theorem}\label{2.1}
Let $(M,g, e^{-f}dv)$ be a smooth measure space with $Ric_{f}\geq-(n-1)K$ for $K\geq 0$. Fixed point $p\in M$. Assume $u$ is a positive $f$-harmonic function on $B_{p}(2R)$. Denote
\begin{equation*}
\theta(R)\doteqdot \sup_{B_{p}(2R)}|\nabla f|.
\end{equation*}
Then for any $0<\epsilon<1$, we have
\begin{eqnarray*}
\begin{aligned}
\sup_{B_{p}(R)}|\nabla \log u|&\leq \frac{c(n)}{1-\epsilon}(\frac{1}{R}+\frac{K^{1/4}+\theta(R)^{1/2}}{R^{1/2}})+\frac{1}{\sqrt{\epsilon(1-\epsilon)}}\theta(R)\\
&\quad+\frac{(n-1)\sqrt{K}}{\sqrt{1-\epsilon}}.
\end{aligned}
\end{eqnarray*}
\end{theorem}
\begin{proof}
The argument is standard from Cheng and Yau, \cite{Cheng-Yau}, also cf. \cite{Li-Wang}. For the convenience of readers, we provided the detail of the computation here.

Set $h=\log u$, cutoff function $\phi$, $G=\phi^{2}|\nabla h|^{2}$. By the Bochner formula for $f$-Laplacian, since $\Delta_{f}h=\frac{\Delta_{f}u}{u}-|\nabla h|^{2}=-|\nabla h|^{2}$,
\begin{eqnarray}
\begin{aligned}
\frac{1}{2}\Delta_{f}|\nabla h|^{2}&= |h_{ij}|^{2}+\langle\nabla u, \nabla(\Delta_{f}h)\rangle+{\rm Ric}_{f}(\nabla h,\nabla h)\\
&\geq |h_{ij}|^{2}-\langle\nabla h, \nabla |\nabla h|^{2}\rangle-(n-1)K|\nabla h|^{2}.
\end{aligned}
\end{eqnarray}
We compute at the point $x$ that $|\nabla h|(x)\neq 0$,  choosing an orthonormal frame $\{e_{i}\}$ at $x$ such that $e_{1}=\frac{\nabla h}{|\nabla h|}$, then we have
\begin{eqnarray*}
\begin{aligned}
|h_{ij}|^{2}&\geq |h_{11}|^{2}+\sum_{\alpha>1}|h_{\alpha\alpha}|^{2}+2\sum_{\alpha>1}|h_{1\alpha}|^{2}\\
&\geq |h_{11}|^{2}+2\sum_{\alpha>1}|h_{1\alpha}|^{2}+\frac{1}{n-1}|\sum_{\alpha>1}h_{\alpha\alpha}|^{2}\\
&=|h_{11}|^{2}+2\sum_{\alpha>1}|h_{1\alpha}|^{2}+\frac{1}{n-1}||\nabla h|^{2}+h_{11}-\langle\nabla f, \nabla h\rangle|^{2}
\end{aligned}
\end{eqnarray*}
Using Cauchy-Schwarz's inequality, for any $1>\epsilon>0$,
\begin{eqnarray*}
\begin{aligned}
&\quad||\nabla h|^{2}+h_{11}-\langle\nabla f, \nabla h\rangle|^{2}\\
&=|\nabla h|^{4}+h_{11}^{2}+2h_{11}|\nabla h|^{2}-2\langle\nabla f,\nabla h\rangle|\nabla h|^{2}-2\langle\nabla f,\nabla h\rangle h_{11}\\
&\quad+\langle\nabla f,\nabla h\rangle^{2}\\
&\geq |\nabla h|^{4}+h_{11}^{2}+2h_{11}|\nabla h|^{2}-\epsilon|\nabla h|^{4}-\frac{1}{\epsilon}\langle\nabla f,\nabla h\rangle^{2}\\
&\quad-h_{11}^{2}-\langle\nabla f,\nabla h\rangle^{2}+\langle\nabla f,\nabla h\rangle^{2}\\
&\geq(1-\epsilon)|\nabla h|^{4}+2h_{11}|\nabla h|^{2}-\frac{1}{\epsilon}|\nabla f|^{2}|\nabla h|^{2}.
\end{aligned}
\end{eqnarray*}
So
\begin{eqnarray*}
\begin{aligned}
|h_{ij}|^{2}&\geq\sum h_{1i}^{2}+\frac{1-\epsilon}{n-1}|\nabla h|^{4}\\
&\quad +\frac{2}{n-1}h_{11}|\nabla h|^{2}-\frac{1}{(n-1)\epsilon}|\nabla f|^{2}|\nabla h|^{2}.
\end{aligned}
\end{eqnarray*}

On the other hand, notice that
\begin{equation*}
\langle\nabla|\nabla h|^{2}, \nabla h\rangle=2h_{ij}h_{i}h_{j}=2h_{11}|\nabla h|^{2},
\end{equation*}
and
\begin{equation*}
|\nabla|\nabla h|^{2}|^{2}=4\sum|h_{ij}h_{j}|^{2}=4\sum h_{1i}^{2}|\nabla h|^{2},
\end{equation*}
which imply
\begin{eqnarray*}
\begin{aligned}
|h_{ij}|^{2}
&\geq \frac{1}{4}|\nabla h|^{-2}|\nabla |\nabla h|^{2}|^{2}\\
&\quad+\frac{1}{n-1}((1-\epsilon)|\nabla h|^{4}+\langle\nabla|\nabla h|^{2},\nabla h\rangle-\frac{1}{\epsilon}|\nabla f|^{2}|\nabla h|^{2}).
\end{aligned}
\end{eqnarray*}
Hence
\begin{eqnarray}\label{2.2}
\begin{aligned}
\frac{1}{2}\Delta_{f}|\nabla h|^{2}&\geq \frac{1}{4}|\nabla h|^{-2}|\nabla |\nabla h|^{2}|^{2}+\frac{1-\epsilon}{n-1}|\nabla h|^{4}-(n-1)K|\nabla h|^{2}\\
&\quad-\frac{n-2}{n-1}\langle\nabla |\nabla h|^{2}, \nabla h\rangle-\frac{1}{(n-1)\epsilon}|\nabla f|^{2}|\nabla h|^{2}.
\end{aligned}
\end{eqnarray}
Since
\begin{equation*}
\Delta_{f}G=(\Delta_{f}\phi^{2})|\nabla h|^{2}+\phi^{2}\Delta_{f}|\nabla h|^{2}+2\langle\nabla \phi^{2},\nabla|\nabla h|^{2}\rangle,
\end{equation*}
by (\ref{2.2}),
\begin{eqnarray*}
\begin{aligned}
&\quad\frac{1}{2}\Delta_{f}G\\
&\geq \frac{1}{4}\phi^{4}G^{-1}|\nabla(\phi^{-2}G)|^{2}+\frac{1-\epsilon}{n-1}\phi^{-2}G^{2}-(n-1)KG\\
&\quad-\frac{n-2}{n-1}\phi^{2}\langle\nabla (\phi^{-2}G),\nabla h\rangle-\frac{1}{(n-1)\epsilon}|\nabla f|^{2}G \\
&\quad+\frac{1}{2}\phi^{-2}(\Delta_{f}\phi^{2})G+\langle\nabla \phi^{2}, \nabla(\phi^{-2}G)\rangle
\end{aligned}
\end{eqnarray*}
At the maximum point $x_{0}$, $\nabla G(x_{0})=0, \Delta G(x_{0})\leq 0$ and $\phi(x_{0})>0$, we have
\begin{eqnarray*}
\begin{aligned}
0\geq& \frac{1-\epsilon}{n-1}G^{2}-(n-1)K\phi^{2}G-3|\nabla \phi|^{2}G+\frac{2(n-2)}{n-1}\phi\langle\nabla\phi,\nabla h\rangle G\\
&-\frac{1}{(n-1)\epsilon}|\nabla f|^{2}\phi^{2}G+\frac{1}{2}(\Delta_{f}\phi^{2})G.
\end{aligned}
\end{eqnarray*}

Since $\phi\langle\nabla\phi,\nabla h\rangle\geq -|\nabla \phi|G^{1/2}$, then
\begin{eqnarray*}
\begin{aligned}
0\geq& (1-\epsilon)G^{2}-(n-1)^{2}K\phi^{2}G-3(n-1)|\nabla \phi|^{2}G-2(n-2)|\nabla\phi|G^{3/2}\\
&-\frac{1}{\epsilon}|\nabla f|^{2}\phi^{2}G+\frac{n-1}{2}(\Delta_{f}\phi^{2})G.
\end{aligned}
\end{eqnarray*}
This can be written as
\begin{eqnarray}
\begin{aligned}
(1-\epsilon)G&\leq-(n-1)\phi\Delta_{f}\phi+2(n-1)|\nabla\phi|^{2}\\
&\quad+(\frac{1}{\epsilon}|\nabla f|^{2}+(n-1)^{2}K)\phi^{2}+2(n-2)|\nabla\phi|G^{1/2}.
\end{aligned}
\end{eqnarray}
Then
\begin{eqnarray*}
\begin{aligned}
(G^{1/2}-\frac{n-2}{1-\epsilon}|\nabla \phi|)^{2}&\leq \frac{1}{1-\epsilon}(-(n-1)\phi\Delta_{f}\phi+2(n-1)|\nabla\phi|^{2})\\
&\quad+\frac{(n-2)^{2}}{(1-\epsilon)^{2}}|\nabla\phi|^{2}+\frac{1}{\epsilon(1-\epsilon)}|\nabla f|^{2}\phi^{2}+\frac{(n-1)^{2}K}{1-\epsilon}\phi^{2}.
\end{aligned}
\end{eqnarray*}
Hence
\begin{eqnarray}\label{G}
\begin{aligned}
G^{1/2}&\leq \sqrt{\frac{1}{1-\epsilon}(-(n-1)\phi\Delta_{f}\phi+2(n-1)|\nabla\phi|^{2})}+\frac{2(n-2)}{1-\epsilon}|\nabla \phi|\\
&\quad+\frac{1}{\sqrt{\epsilon(1-\epsilon)}}|\nabla f|+\frac{(n-1)\sqrt{K}}{\sqrt{1-\epsilon}},
\end{aligned}
\end{eqnarray}
where we have used the inequality $\sqrt{A+B}\leq \sqrt{A}+\sqrt{B}$ for any $A\geq 0, B\geq 0$ and $\phi\leq 1$.

We consider the cutoff function as following, $\phi(x)=\eta(\frac{d(x)}{R})$, $\eta$ is the same in the proof of the Theorem \ref{DivergenceThm}. Using the comparison theorem was proved by Wei and Wylie (\cite{Wei-Wylie}, Theorem 1.1), we have
\begin{equation}
\Delta_{f}d(x)\leq (n-1)(\frac{1}{d(x)}+\sqrt{K})+\theta(R).
\end{equation}
Then $|\nabla\phi|\leq \frac{4}{R}$, and
\begin{eqnarray*}
\begin{aligned}
-\Delta_{f}\phi&=-\frac{\eta'}{R}\Delta_{f}d(x)-\frac{\eta''}{R^{2}}\\
&\leq \frac{4}{R}((n-1)(\frac{1}{R}+\sqrt{K})+\theta(R))+\frac{4}{R^{2}}.
\end{aligned}
\end{eqnarray*}
Hence by (\ref{G}), we get
\begin{eqnarray*}
\begin{aligned}
G^{1/2}&\leq \frac{c(n)}{\sqrt{1-\epsilon}}(\frac{1}{R}+\frac{K^{1/4}}{R^{1/2}}+\frac{\theta(R)^{1/2}}{R^{1/2}})+\frac{8n}{(1-\epsilon)R}\\
&\quad+\frac{1}{\sqrt{\epsilon(1-\epsilon)}}\theta(R)+\frac{(n-1)\sqrt{K}}{\sqrt{1-\epsilon}}\\
&\leq \frac{c(n)}{1-\epsilon}(\frac{1}{R}+\frac{K^{1/4}+\theta(R)^{1/2}}{R^{1/2}})+\frac{1}{\sqrt{\epsilon(1-\epsilon)}}\theta(R)\\
&\quad+\frac{(n-1)\sqrt{K}}{\sqrt{1-\epsilon}}
\end{aligned}
\end{eqnarray*}
Hence
\begin{eqnarray*}
\begin{aligned}
\sup_{B_{p}(R)}|\nabla \log u|&\leq \frac{c(n)}{1-\epsilon}(\frac{1}{R}+\frac{K^{1/4}+\theta(R)^{1/2}}{R^{1/2}})+\frac{1}{\sqrt{\epsilon(1-\epsilon)}}\theta(R)\\
&\quad+\frac{(n-1)\sqrt{K}}{\sqrt{1-\epsilon}}.
\end{aligned}
\end{eqnarray*}
\end{proof}

\begin{remark}
If $\theta(R)=0$, we can take $\epsilon=0$. Then
\begin{equation*}
\sup_{B_{p}(R)}|\nabla \log u|\leq c(n)(\frac{1}{R}+\frac{K^{1/4}}{R^{1/2}})+(n-1)\sqrt{K}.
\end{equation*}
\end{remark}
\begin{theorem}\label{LocalGradientEstimate1}(Theorem \ref{LocalGradientEstimate})
Let $(M,g,f)$ be a complete gradient shrinking Ricci soliton, assume $c_{1}, c_{2}$ are the constants in Lemma 1.1, let $c_{3}=c_{1}+c_{2}$, then if $u$ is a positive $f$-harmonic function on $B_{p}(3R)$ for $R>c_{3}$, then we have
\begin{equation}
\sup_{x\in B_{p}(R)}|\nabla \log u|(x)\leq \frac{20n}{R}+2R+2c_{2}+c_{3}.
\end{equation}
\end{theorem}
\begin{proof}
Same computation as in the proof of Theorem \ref{2.1}, and for shrinking Ricci solitons we take $K=0$. Here we take cut-off function $\phi(x)=\eta(\frac{\rho(x)}{r})$, where $\eta(t)$ is the same as in the proof of Theorem \ref{2.1}.  The support set of $\phi(x)$ is contained in $B_{p}(r_{1})$, here $r_{1}=2r+c_{1}$, we choose $R$ such that $r_{1}\leq 3R$.

Let $h=\log f$. Consider $G:B_{p}(r_{1})\rightarrow \mathbb{R}$, $G=\phi^{2}|\nabla h|^{2}$. Since $G$ is nonnegative on $B_{p}(r_{1})$ and $G=0$ on $\partial B_{p}(r_{1})$, it follows that $G$ attains a maximum point in the interior of $B_{p}(r_{1})$. Let $x_{0}$ be this maximum point. By the maximum principle,
\begin{eqnarray}
\begin{aligned}
&\nabla G(x_{0})=0,\\
&\Delta G(x_{0})\leq 0
\end{aligned}
\end{eqnarray}

For $r\leq \rho\leq 2r$, using (\ref{EstimateRho}), we have
\begin{eqnarray*}
\begin{aligned}
-\Delta_{f}\phi(x)&=-\frac{\eta'}{r}\Delta _{f}\rho(x)-\frac{\eta''}{r^{2}}|\nabla \rho(x)|^{2}\\
&\leq \frac{2n}{r\sqrt{f}}+\frac{4}{r^{2}}\\
&\leq \frac{4n+4}{r^{2}}.
\end{aligned}
\end{eqnarray*}
Since $|\nabla f|^{2}\leq f\leq \frac{1}{4}(d(x)+c_{2})^{2}\leq \frac{1}{4}(2r+c_{1}+c_{2})^{2}$, taking $K=0$ in (\ref{G}), and using $|\nabla \phi|\leq \frac{4}{r}$, then we have
\begin{equation*}
G^{1/2}\leq \frac{10n}{(1-\epsilon)r}+\frac{1}{\sqrt{\epsilon(1-\epsilon)}}(r+\frac{c_{3}}{2}).
\end{equation*}

If $d(x)\leq r-c_{2}$, $\rho(x)\leq r$, then $\phi(x)=1$. Hence if we choose $r=R+c_{2}$, then
\begin{equation}
\sup_{B_{p}(R)}|\nabla h|\leq \frac{10n}{(1-\epsilon)(R+c_{2})}+\frac{1}{\sqrt{\epsilon(1-\epsilon)}}(R+c_{2}+\frac{c_{3}}{2}).
\end{equation}
Take $\epsilon=\frac{1}{2}$, then
\begin{equation}
\sup_{B_{p}(R)}|\nabla h|\leq \frac{20n}{R}+2R+2c_{2}+c_{3}.
\end{equation}
\end{proof}

By Theorem \ref{LocalGradientEstimate1}, we can obtain the Harnack inequality for complete shrinking gradient Ricci solitons.
\begin{corollary}\label{HarnackInequality1}(Corollary \ref{HarnackInequality})
Let $(M,g,f)$ be a complete shrinking gradient Ricci soliton. Fixed point $p \in M$, assume $u$ is a positive $f$-harmonic function on $M$, then
\begin{equation}
u(x)\leq u(p)e^{3d(x)^{2}}
\end{equation}
for $d(x)\geq \max\{2(2c_{2}+c_{3}), 10n\}$. Here $d(x)$ means the distance function from $x$ to $p$, $c_{2}, c_{3}$ are the constants in Theorem \ref{LocalGradientEstimate1}.
\end{corollary}
\begin{proof}
For any $x\in M$ with $d(x)\geq \max\{2c_{3}, 10n\}$, connecting $x$ and $p$ by a minimal geodesic $\gamma(t)$, by the triangle inequality, we know $\gamma(t)$ is contained in $B_{p}(d(x))$. By Theorem \ref{LocalGradientEstimate1}, we have
$$\sup_{y\in B_{p}(d(x))}|\nabla \log u(y)|\leq 3d(x).$$
So
\begin{eqnarray*}
\begin{aligned}
\log u(x)-\log u(p)&=\int_{0}^{d(x)}\frac{d}{dt}\log u(\gamma (t))dt\\
&=\int_{0}^{d(x)}\langle\nabla \log u, \dot{\gamma}(t)\rangle dt\\
&\leq \int_{0}^{d(x)}|\nabla \log u|dt\\
&\leq 3d(x)^{2}.
\end{aligned}
\end{eqnarray*}
Hence
$$u(x)\leq u(p)e^{3d(x)^{2}}$$
for any $x$ that $d(x)\geq \max\{2(2c_{2}+c_{3}), 10n\}$.
\end{proof}
\begin{remark}
Using the same proof, we can also get a local version of Harnack estimate for the positive $f$-harmonic functions on $B_{p}(3R)$ for $R> C_{1}$, here $C_{1}$ is a uniform constant. Then there exists a constant $C_{2}$ such that for any $x\in B_{p}(R)$, we have
\begin{equation*}
u(x)\leq C_{2}u(p)e^{3d(x)^{2}}.
\end{equation*}
\end{remark}
\section{Proof of Theorem \ref{MainThm}}
Now we can prove Theorem \ref{MainThm}.
\begin{proof}[Proof of Theorem \ref{MainThm}]
Fixed a point $p\in M$. Set $h=\log(1+u)$, then $\Delta_{f}h=-|\nabla h|^{2}\leq 0$. By Corollary \ref{HarnackInequality1}, there exists a positive constant $C$, such that $0<h(x)\leq C(d(x)+1)^{2}$. Using the facts
$$\frac{1}{4}(d(x)-c_{1})^{2}\leq f(x)\leq \frac{1}{4}(d(x)+c_{2})^{2}$$
and the volume growth of $M$ is polynomial with degree at most $n$, we have
$$\int_{M}h(x)e^{-f}dv<+\infty.$$
Then by the Theorem \ref{DivergenceThm}, we have
\begin{equation*}
\Delta_{f}(h)=0.
\end{equation*}
Hence
\begin{equation*}
|\nabla h|^{2}=0.
\end{equation*}
So $h$ is constant and $u$ is constant.
\end{proof}

Then we can obtain the Corollary \ref{LpLiouvilleThm} and Corollary \ref{LpLiouvilleThm1}.
\begin{proof}[Proof of Corollary \ref{LpLiouvilleThm} and Corollary \ref{LpLiouvilleThm1}]
Since
$$\Delta_{f}u^{p}=pu^{p-1}\Delta_{f}u+p(p-1)u^{p-2}|\nabla u|^{2},$$
if $p\geq 1$ and $u\geq 0$, then $\Delta_{f}u\geq 0$ implies
$$\Delta_{f}u^{p}\geq 0.$$
If $0<p\leq 1$ and $u>0$, then $\Delta_{f}u\leq 0$ implies
$$\Delta_{f}u^{p}\leq 0.$$
Then by Theorem \ref{MainThm},  we know $u$ is constant.
\end{proof}

\section*{Acknowledgements}
The first author would like to thank Professor Gang Tian and Yuguang Shi for persistent encouragement, his research is partially supported by NSFC No. 11501027 and the Fundamental Research Funds for the Central Universities (No. 2015JBM103, 2014RC028 and 2016JBM071). The work is carried out during the second author's visit at Rutgers University. The second author would like to thank Math Department of Rutgers University for its hospitality.  He is partially supported by NSFC No. 11301017, Research Fund for the Doctoral Program of Higher Education of China, the Fundamental Research Funds for the Central Universities and the Scholarship from China Scholarship Council. Both authors would like to thank Jiayong Wu for his useful suggestions.

\bibliographystyle{amsplain}

\end{document}